\documentclass[letterpaper]{amsart}

\usepackage[utf8]{inputenc}
\usepackage{lmodern}
\usepackage[T1]{fontenc}
\usepackage{amssymb}
\usepackage{mathtools}
\usepackage{tikz-cd}
\usepackage{hyperref}

\def\doi#1{\href{http://dx.doi.org/#1}{doi:#1}}

\makeatletter
\newcommand\niton{\mathrel{\m@th\mathpalette\canc@l\owns}}
\newcommand\canc@l[2]{{\ooalign{$\hfil#1/\mkern1mu\hfil$\crcr$#1#2$}}}
\makeatother

\theoremstyle{plain}
\newtheorem{theorem}{Theorem}[section]
\newtheorem{proposition}[theorem]{Proposition}
\newtheorem{lemma}[theorem]{Lemma}
\newtheorem{corollary}[theorem]{Corollary}

\theoremstyle{definition}

\theoremstyle{remark}
\newtheorem*{acknowledgements}{Acknowledgements}

\numberwithin{equation}{section}

\DeclareMathAlphabet\mathbfit{OML}{cmm}{b}{it}

\let\newterm\emph

\def\Z{\mathbb Z}
\def\Q{\mathbb Q}
\def\R{\mathbb R}
\def\C{\mathbb C}
\def\F{\mathbb F}
\def\CC{\mathcal C}
\def\DD{\mathcal D}
\def\kk{\Bbbk}

\DeclareMathOperator{\supp}{supp}
\DeclareMathOperator{\codim}{codim}
\DeclareMathOperator{\corank}{corank}
\DeclareMathOperator{\coker}{coker}
\def\aa{\mathbfit a}

\def\cc{\mathbfit c}

\def\deg#1{|#1|}

\def\lambda{\ell}

\def\cf{\emph{cf.}}

\def\hatgamma{\skew2\hat\gamma}

\def\RZtwo{H^{2}(BT;\Z)}

\def\HT{H_{T}}
\def\HG{H_{G}}

\begin{document}

\title{The syzygy order of big polygon spaces}
\author{Matthias Franz}
\thanks{M.\,F.\ was supported by an NSERC Discovery Grant.}
\address{Department of Mathematics, University of Western Ontario,
  London, Ont.\ N6A\;5B7, Canada}
\email{mfranz@uwo.ca}
\author{Jianing Huang}
\email{jhuan233@uwo.ca}
      

\begin{abstract}
  Big polygon spaces are compact orientable manifolds with a torus action
  whose equivariant cohomology can be torsion-free or reflexive
  without being free as a module over~\(H^{*}(BT)\).
  We determine the exact syzygy order of the equivariant cohomology
  of a big polygon space in terms of the length vector defining it.
  The proof uses a refined characterization of syzygies in terms
  of certain linearly independent elements in~\(H^{2}(BT)\)
  adapted to the isotropy groups occurring in a given \(T\)-space.
\end{abstract}

\hypersetup{pdfauthor=\authors}

\maketitle

\section{Introduction}

Let \(T\cong(S^{1})^{r}\) be a torus,
and let \(X\) be a \(T\)-manifold whose cohomology~\(H^{*}(X)\) (with real coefficients) is finite-dimensional.
A powerful tool to compute the equivariant cohomology~\(\HT^{*}(X)\) is the Chang--Skjelbred sequence
\begin{equation}
  \label{eq:chang-skjelbred}
  0 \longrightarrow
  \HT^{*}(X) \longrightarrow
  \HT^{*}(X^{T}) \longrightarrow
  \HT^{*+1}(X_{1},X^{T})
\end{equation}
where \(X^{T}\subset X\) is the fixed point set
and the equivariant \(1\)-skeleton~\(X_{1}\) the union of orbits of dimension at most~\(1\).
The first map is induced by the inclusion~\(X^{T}\hookrightarrow X\) and the second one
is the connecting homomorphism in the long exact sequence for the pair~\((X_{1},X^{T})\).

If \(\HT^{*}(X)\) is free as a module over the polynomial ring~\(R=H^{*}(BT)\),
then the Chang--Skjelbred sequence is exact \cite[Prop.~2.4]{ChangSkjelbred:1974},
which implies that \(\HT^{*}(X)\) can be calculated out of the equivariant \(1\)-skeleton.
In many cases of interest, \(X^{T}\) is finite and \(X_{1}\) a union of \(2\)-spheres glued together at their poles.
In such a setting, this approach is called the GKM~method after work of Goresky--Kottwitz--Mac\-Pherson~\cite[Thm.~7.2]{GoreskyKottwitzMacPherson:1998}.

It is not hard to find examples of \(T\)-manifolds such that \eqref{eq:chang-skjelbred} is exact
without \(\HT^{*}(X)\) being free over~\(R\), see below. This phenomenon was studied in detail
by Allday--Franz--Puppe~\cite{AlldayFranzPuppe:orbits1},~\cite{AlldayFranzPuppe:orbits4},
who characterized those \(T\)-manifolds for which the Chang--Skjelbred sequence is exact;
in~\cite{Franz:nonab} this is generalized to non-abelian Lie groups. Allday--Franz--Puppe
actually proved a more general theorem that involves higher equivariant skeletons.
For~\(-1\le i\le r\),
we write \(X_{i}\subset X\) for 
the union of orbits of dimension at most~\(i\),
so that \(X_{0}=X^{T}\) and~\(X_{r}=X\).

\begin{theorem}[Allday--Franz--Puppe]
  The Chang--Skjelbred sequence~\eqref{eq:chang-skjelbred} is exact
  if and only if \(\HT^{*}(X)\) is a reflexive \(R\)-module.
  More generally, for any~\(1\le k\le r\), the 
  sequence
  \begin{multline*}
    0 \longrightarrow
    \HT^{*}(X) \longrightarrow
    \HT^{*}(X_{0}) \longrightarrow
    \HT^{*+1}(X_{1},X_{0}) \\ \longrightarrow
    \HT^{*+2}(X_{2},X_{1}) \longrightarrow
    \dots \longrightarrow
    \HT^{*+k-1}(X_{k-1},X_{k-2})
  \end{multline*}
  is exact if and only if \(\HT^{*}(X)\) is a \(k\)-th syzygy over~\(R\).
\end{theorem}

See~\cite[Thm.~1.1]{AlldayFranzPuppe:orbits1}.
The additional maps in the sequence above are the connecting homomorphisms for the triples~\((X_{i+1},X_{i},X_{i-1})\).

Recall that an \(R\)-module is reflexive if the canonical map to its double-dual is an isomorphism.
Syzygies are a notion from commutative algebra that interpolates between torsion-free and free modules,
see Section\nobreakspace \ref {sec:char-syz} for the precise definition.
The first syzygies over~\(R\) are exactly the torsion-free modules,
the second syzygies the reflexive ones and the \(r\)-th syzygies the free ones.

As a corollary
(\cf~the comment following~\cite[Thm.~5.7]{AlldayFranzPuppe:orbits1})
we get the result of Atiyah and Bredon~\cite[Thm.]{Bredon:1974} that the sequence
\begin{multline}
  0 \longrightarrow
  \HT^{*}(X) \longrightarrow
  \HT^{*}(X_{0}) \longrightarrow
  \HT^{*+1}(X_{1},X_{0}) \\ \longrightarrow
  \HT^{*+2}(X_{2},X_{1}) \longrightarrow
  \dots \longrightarrow
  \HT^{*+r}(X_{r},X_{r-1}) \longrightarrow
  0
\end{multline}
is exact if and only if \(\HT^{*}(X)\) is free over~\(R\),
which strengthens the Chang--Skjelbred theorem.

It is not difficult to construct \(T\)-manifolds such that \(\HT^{*}(X)\) is a \(k\)-th syzygy for~\(k<r\).
For example,
the usual rotation action of~\(S^{1}\) on~\(S^{2}\) gives an action of~\(T\) on~\((S^{2})^{r}\)
such that \(\HT^{*}(X)\) is free over~\(R\).
By suitably removing two fixed points, any syzygy order less than~\(r\) can be realized \cite[Sec.~6.1]{AlldayFranzPuppe:orbits1}.

The situation becomes much more intriguing if one looks at compact orientable \(T\)-manifolds.
For such an~\(X\), another result of Allday--Franz--Puppe says that
if \(\HT^{*}(X)\) is a syzygy of order~\(\ge r/2\),
then it is actually free over~\(R\)
\cite[Cor.~1.4]{AlldayFranzPuppe:orbits1}.

It already appears very difficult to construct compact orientable \(T\)-manifolds
such that \(\HT^{*}(X)\) is torsion-free, but not free over~\(R\). The first such examples
were the `mutants of compactified representations' given by Franz--Puppe in~2008 \cite[Sec.~4]{FranzPuppe:2008}.
Recently, the first author found a family of \(T\)-manifolds,
the so-called big polygon spaces, that generalize one of the mutants
to arbitrary syzygies \cite{Franz:maximal}. We recall the definition.

Let \(\ell\in\R^{r}\), called a \newterm{length vector} in this context.
We assume that \(\ell\) is \newterm{generic}, meaning that
\begin{equation}
  \sum_{j\in J}\ell_{j} \ne \sum_{j\notin J}\ell_{j}
\end{equation}
for any subset~\(J\subset\{1,\dots,r\}\).
Depending on which side dominates, \(J\) is called \newterm{\(\ell\)-long }or \newterm{\(\ell\)-short}.

Let \(p\),~\(q\ge1\).
The \newterm{big polygon space}~\(X(\ell)=X_{p,q}(\ell)\) is the real algebraic subvariety of~\(\C^{(p+q)r}\)
defined by the equations
\begin{equation}
  \left\{\:
  \begin{aligned}
  \|u_{j}\|^{2} + \|z_{j}\|^{2} &= 1 && \text{for any~\(1\le j\le r\),} \\
  \ell_{1}u_{1}+ \dots + \ell_{r}u_{r} &= 0\mathrlap{,}
  \end{aligned}
  \right.
\end{equation}
where~\(u_{1}\),~\dots,~\(u_{r}\in\C^{p}\) and~\(z_{1}\),~\dots,~\(z_{r}\in\C^{q}\).
Since \(\ell\) is generic,
\(X(\ell)\) is a compact orientable manifold with a smooth action of~\(T=(S^{1})^{r}\)
given by scalar multiplication of the \(z\)-variables,
\begin{equation}
  g\cdot(u,z)
  = (u, g_{1}z_{1},\dots,g_{r}z_{r}),
\end{equation}
see~\cite[Lemma~2.1]{Franz:maximal}.
The fixed point set~\(X(\ell)^{T}\) is the `space of polygons'~\(E_{2p}(\ell)\) studied by Farber--Fromm~\cite{FarberFromm:2013}.

It turns out that \(\HT^{*}(X(\ell))\) is never free over~\(R\).
In fact, \(\HT^{*}(X(\ell))\) is not a syzygy of order
\begin{equation}
  \mu(\ell) = \min\{\,\sigma_{\ell}(J) \mid \text{\(J\subset\{1,\dots,r\}\) is \(\ell\)-long and \(\sigma_{\ell}(J)>0\)}\,\}
\end{equation}
where
\begin{equation}
  \sigma_{\ell}(J) = \# \{\, j\in J \mid \text{\(J\setminus j\) is \(\ell\)-short}\,\},
\end{equation}
see~\cite[Prop.~6.3]{Franz:maximal}.
Our main result confirms the conjecture made in~\cite[Conj.~6.6]{Franz:maximal}.

\begin{theorem}
  \label{thm:main}
  The syzygy order of~\(\HT^{*}(X(\ell))\) over~\(R\) equals \(\mu(\ell)-1\).
\end{theorem}

Our proof of Theorem\nobreakspace \ref {thm:main} is purely algebraic and uses the description
of~\(\HT^{*}(X(\ell))\) given in~\cite[Lemma~4.4]{Franz:maximal}.
It is inspired by the proof appearing in
the second author's Ph.\,D.~thesis~\cite{Huang:2018},
which in turn is based on the quotient criterion for syzygies developed in~\cite{Franz:orbits3}
and on Morse--Bott theory for manifolds with corners.

The largest possible syzygy order for~\(r=2m+1\) and~\(r=2m+2\) is \(m\).
It is known that this syzygy order is realized
by an essentially unique length vector
which for odd~\(r\) corresponds to the equilateral case \(\ell=(1,\dots,1)\),
see~\cite[Cor.~6.4]{Franz:maximal} and also Proposition\nobreakspace \ref {thm:syzymy-m-characterization}.
From Theorem\nobreakspace \ref {thm:main} we deduce that syzygies of the next smaller order are also unique or at least almost unique.

\begin{corollary}
  \label{thm:char-syzygy-intro}
  Let \(r\ge3\), and let \(\ell\in\R^{r}\) be a generic length vector with weakly increasing non-negative components.
  \begin{enumerate}
  \item Assume that \(r=2m+1\) is odd. Then \(\HT^{*}(X(\ell))\) is a syzygy of order~\(m-1\) if and only if
    \(X(\ell)\) is equivariantly diffeomorphic to~\(X(0,0,1,\dots,1)\).
  \item Assume that \(r=2m+2\) is even. Then \(\HT^{*}(X(\ell))\) is a syzygy of order~\(m-1\) if and only if
    \(X(\ell)\) is equivariantly diffeomorphic to~\(X(0,0,0,1,\dots,1)\) or to~\(X(1,1,1,2,\dots,2)\).
  \end{enumerate}
\end{corollary}

To relate our algebraic reasoning with equivariant cohomology, we develop a refined criterion
for syzygies in equivariant cohomology which is of independent interest. It involves the notion
of a \newterm{\(k\)-localizing subset}~\(S\subset H^{2}(BT)\) for a given `nice' \(T\)-space~\(X\),
see again Section\nobreakspace \ref {sec:char-syz} for the definitions. For a big polygon space~\(X(\ell)\),
the set~\(\{t_{1},\dots,t_{r}\}\) of indeterminates of~\(R\) is \(k\)-localizing for any~\(k\).

\begin{theorem}
  \label{thm:refined-crit-intro}
  Let \(S\subset H^{2}(BT;\Z)\) be \(k\)-localizing for~\(X\) for some~\(k\ge 1\).
  Then \(\HT^{*}(X)\) is a \(k\)-th syzygy
  over~\(R\) if and only if any linearly independent
  sequence in~\(S\) of length at most~\(k\) is \(\HT^{*}(X)\)-regular.
\end{theorem}

The proof of Theorem\nobreakspace \ref {thm:refined-crit-intro} appears in Section\nobreakspace \ref {sec:char-syz}.
Theorem\nobreakspace \ref {thm:main} is proven in Section\nobreakspace \ref {sec:big-polygon} and Corollary\nobreakspace \ref {thm:char-syzygy-intro}
in Section\nobreakspace \ref {sec:high-syzygies}. In Section\nobreakspace \ref {sec:real} we state versions of our results
for actions of \(2\)-tori~\((\Z_{2})^{r}\) and certain `real' analogues of big polygon spaces
which have recently been studied by Puppe~\cite{Puppe:2018}.

\begin{acknowledgements}
  We thank Dirk Schütz for sharing his Java code to enumerate all chambers of generic length vectors
  in a given dimension. It enabled us to verify the syzygy order conjecture in dimension~\(10\).
\end{acknowledgements}

\section{A refined characterization of syzygies}
\label{sec:char-syz}

From now on,
all cohomology is taken with coefficients in a field~\(\kk\) of characteristic~\(0\)
unless stated otherwise.
Let \(T\cong(S^{1})^{r}\) be a torus; we write \(R=H^{*}(BT)\).

Recall that any element~\(a\in \RZtwo\cong H^{1}(T;\Z)\) can be interpreted
as a character~\(\chi_{a}\colon T\to S^{1}\).
We write \(t_{i}\in\RZtwo\) for the element corresponding
to the \(i\)-th coordinate~\(T\to S^{1}\), so that \(R=\kk[t_{1},\dots,t_{r}]\).
For any linearly independent sequence~\(\aa=(a_{1},\dots,a_{m})\) in~\(\RZtwo\)
we write \(T(\aa)\subset T\) for the identity component of the intersection
of~\(\ker\chi_{a_{1}}\),~\dots,~\(\ker\chi_{a_{m}}\), which is of codimension~\(m\).
Given an \(R\)-module~\(M\), we also write \(M/\aa=M/(a_{1},\dots,a_{m})M\).

Let \(X\) be a \(T\)-space.
We say that \(X\) is \newterm{nice}
if it is Hausdorff, second-countable, locally compact and locally contractible,
see~\cite[Sec.~3.1]{AlldayFranzPuppe:orbits1}.
For instance, \(X\) can be a \(T\)-manifold or \(T\)-orbifold or a complex algebraic variety with an algebraic action of~\((\C^{\times})^{r}\).
We additionally assume that \(H^{*}(X)\) is finite-dimensional and that only finitely many subtori of~\(T\)
occur as identity components of isotropy groups in~\(X\). In the examples just mentioned,
this last condition is redundant, see~\cite[Thm.~7.7]{Franz:nonab}.
\footnote{The algebraic case reduces to the manifold case: Algebraic varieties
  have finite Betti sum and can be decomposed into finitely many smooth varieties,
  stable with respect to an algebraic action.}
(We have implicitly used this in the introduction already.)

Let \(k\ge0\). A finitely generated \(R\)-module~\(M\) is called a \newterm{\(k\)-th syzygy} 
if any regular sequence in~\(R\) of length at most~\(k\) is also \(M\)-regular.
(See~\cite[Sec.~2.3]{AlldayFranzPuppe:orbits1} for equivalent definitions of syzygies.)
If \(M\) is a syzygy of order~\(k\), but not of order~\(k+1\), then we say that the \newterm{syzygy order} of~\(M\) equals \(k\).

In our topological context, it is enough to consider sequences of linear elements.

\begin{lemma}
  \label{thm:syzygy-lin-indep-seq}
  Let \(X\) be a nice \(T\)-space and let \(k\ge0\). Then \(\HT^{*}(X)\) is a \(k\)-th syzygy
  if and only if every linear independent sequence in~\(\RZtwo\) of length at most~\(k\)
  is \(\HT^{*}(X)\)-regular.
\end{lemma}

\begin{proof}
  This is implicit in~\cite[Thm.~5.7]{AlldayFranzPuppe:orbits1}. There it is shown
  that \(\HT^{*}(X)\) is a \(k\)-th syzygy if and only if it is free over all subrings~\(H^{*}(BT'')\subset R\)
  where \(T''\) is a quotient of~\(T\) of rank~\(\le k\) (equivalently, equal to~\(k\) if~\(k\le r\)).
  These are exactly the subrings of~\(R\) that are generated by linearly independent sequences in~\(\RZtwo\) of length~\(\le k\).
  
  Let \(\aa\) be such a sequence and let \(T''\) be the corresponding quotient of~\(T\).
  Because the graded module~\(M=\HT^{*}(X)\) is bounded below,
  it is free over~\(R''=H^{*}(BT'')\) if and only if \(\aa\) is \(M\)-regular,
  see~\cite[Lemma p.~5]{Hochster:CM}.
  (The argument given there remains valid for non-finitely generated modules.)
\end{proof}

\begin{lemma}
  \label{thm:reduction}
  Let \(X\) be a \(T\)-space.
  If a sequence~\(\aa\) in~\(\RZtwo\) is \(\HT^{*}(X)\)-regular,
  then the restriction map~\(\HT^{*}(X)\to H^{*}_{T(\aa)}(X)\) induces an isomorphism
  \begin{equation*}
    H^{*}_{T(\aa)}(X) \cong \HT^{*}(X)/\aa.
  \end{equation*}
\end{lemma}

\begin{proof}
  This is again contained in the proof of~\cite[Thm.~5.7]{AlldayFranzPuppe:orbits1}.
  By induction, we may assume that \(\aa\) consists of a single element~\(0\ne a\in\RZtwo\).
  We may also assume that it is not divisible by any integer~\({}>1\).
  
  Let \(C_{T}^{*}(X)\) be the singular Cartan model for~\(X\), \cf~\cite[Sec.~3.2]{AlldayFranzPuppe:orbits1}.
  We may assume that \(a\) is contained in the basis of~\(\RZtwo\) chosen in the definition of the singular Cartan model.\
  We then have a short exact sequence
  \begin{equation}
    0\longrightarrow C_{T}^{*}(X) \stackrel{{}\cdot a}{\longrightarrow}
    C_{T}^{*}(X) \longrightarrow C_{T''}^{*}(X) \longrightarrow 0.
  \end{equation}
  Because \(a\) is \(\HT^{*}(X)\)-regular, this induces the exact sequence
  \begin{equation}
    0\longrightarrow \HT^{*}(X) \stackrel{{}\cdot a}{\longrightarrow}
    \HT^{*}(X) \longrightarrow H_{T''}^{*}(X) \longrightarrow 0,
  \end{equation}
  proving the claim.
\end{proof}

For any~\(x\in X\), the kernel of the restriction map~\(p_{x}\colon H^{2}(BT)\to H^{2}(BT_{x})\)
has dimension equal to the codimension of~\(T_{x}\) in~\(T\).
We say that a subset~\(S\subset H^{2}(BT)\) is \emph{\(k\)-localizing for~\(X\)}
if for any~\(x\in X\) at least~\(\min(k,\codim T_{x})\) linearly independent elements from~\(S\)
lie in~\(\ker p_{x}\).
This notion behaves well with respect to subtori:

\begin{lemma}
  \label{thm:localizing-quotient}
  Let \(X\) be a \(T\)-space, and let \(S\subset H^{2}(BT)\) be \(k\)-localizing for~\(X\)
  for some~\(k\ge0\).
  For any subtorus~\(T'\subset T\) of codimension~\(l\le k\),
  the image of~\(S\) in~\(H^{2}(BT')\) is \((k-l)\)-localizing for~\(X\), considered as a \(T'\)-space.
\end{lemma}

\begin{proof}
  Let \(x\in X\) and consider the commutative diagram of surjections
  \begin{equation}
    \begin{tikzcd}
      H^{2}(BT) \arrow{r}{p_{x}} \arrow{d}[left]{\pi } & H^{2}(BT_{x}) \arrow{d} \\
      H^{2}(BT') \arrow{r}{p'_{x}} & H^{2}(BT'_{x})\rlap{.}
    \end{tikzcd}
  \end{equation}
  If \(S\) contains a basis for~\(\ker p_{x}\), then \(\pi (S)\) contains one for~\(\ker p'_{x}\).
  If \(S\) contains \(k\)~linearly independent elements from~\(\ker p_{x}\), then \(\pi (S)\)
  contains \(k-l\)~linearly independent elements from~\(\ker p'_{x}\). In either case,
  we have found enough linearly independent elements in~\(\pi (S)\cap\ker p'_{x}\),
  which proves the claim.
\end{proof}

\begin{theorem}
  \label{thm:char-syzygy}
  Let \(X\) be a nice \(T\)-space, and
  let \(S\subset\RZtwo\) be \(k\)-localizing for~\(X\) for some~\(k\ge 0\).
  Then \(\HT^{*}(X)\) is a \(k\)-th syzygy
  over~\(R\) if and only if any linearly independent 
  sequence in~\(S\) of length at most~\(k\) is
  \(\HT^{*}(X)\)-regular.
\end{theorem}

\begin{proof}
  The `only if' direction follows from the definition of syzygies given above.
  We prove the converse by induction on~\(k\). Note that we may assume \(0\notin S\).

  We consider first the case~\(k=1\).
  Because \(S\) is \(1\)-localizing for~\(X\), we can, for any~\(x\notin X^{T}\), find an element in~\(S\)
  lying in the kernel of the restriction map \(H^{*}(BT)\to H^{*}(BT_{x})\). By the localization theorem
  in equivariant cohomology~\cite[Thm.~3.2.6]{AlldayPuppe:1993},
  this implies that the bottom arrow in the commutative diagram
  \begin{equation}
    \begin{tikzcd}
      \HT^{*}(X) \arrow{r} \arrow{d} & \HT^{*}(X^{T}) \arrow{d} \\
      \hat S^{-1}\HT^{*}(X) \arrow{r} & \hat S^{-1}\HT^{*}(X^{T})
    \end{tikzcd}
  \end{equation}
  is an isomorphism, where \(\hat S\subset R\) is the multiplicative subset generated by~\(S\).

  By assumption, no element in~\(\hat S\) is a zero-divisor for~\(\HT^{*}(X)\),
  so the left localization map in the diagram is injective.
  It follows that the top arrow is also injective,
  meaning that the equivariant cohomology of~\(X\) embeds into that of the fixed point set.
  Since \(\HT^{*}(X^{T})\cong H^{*}(X^{T})\otimes R\) is free over~\(R\), \(\HT^{*}(X)\) must be torsion-free.

  We now consider the case~\(k>1\) and assume that \(M\coloneqq\HT^{*}(X)\) is not a \(k\)-th syzygy.
  By Lemma\nobreakspace \ref {thm:syzygy-lin-indep-seq}
  this means that there is an \(R\)-regular sequence~\(\aa\) of length at most~\(k\) in~\(\RZtwo\) that is not \(M\)-regular.
  We are going to show that there is another such sequence contained in~\(S\).
  If \(\aa=(a_{1},\dots,a_{m})\) is of length~\(m<k\)
  or if \(m=k\) and~\(\aa'=(a_{1},\dots,a_{k-1})\) is not \(M\)-regular,
  then \(M\) is not a syzygy of order~\(k-1\), and we are done by induction.

  So we can assume \(m=k\) and that \(\aa'\) is \(M\)-regular.
  We write \(T'=T(\aa')\), \(R'=H^{*}(BT')=R/\aa'\)
  as well as \(\pi'\) for the canonical projection~\(R\to R'\).
  By Lemma\nobreakspace \ref {thm:reduction} we have an isomorphism
  \begin{equation}
    M' \coloneqq H_{T'}^{*}(X) \cong M/\aa',
  \end{equation}
  and \(\pi'(a_{k})\in R'\) is a zero-divisor for this module.
  Moreover, Lemma\nobreakspace \ref {thm:localizing-quotient} implies that \(\pi'(S)\subset R'\) is \(1\)-localizing for the \(T'\)-space~\(X\).
  By the already established base case, there is a zero-divisor~\(\pi'(b)\ne0\) in~\(\pi'(S)\),
  hence \(0\ne b\in S\) is also a zero-divisor for~\(M'\).
  Therefore, the sequence~\((\aa',b)\) is not \(M\)-regular.

  We may assume that \(b\) is not a zero-divisor for~\(M\) for otherwise we would be done as \(M\) would not be a first syzygy.
  Because \(M\) is graded and bounded below and the sequence~\((\aa',b)\) is made of homogeneous elements,
  we can rearrange it \cite[Prop., p.~1]{Hochster:CM} to obtain \((b,\aa')\),
  which is again \(R\)-regular, but not \(M\)-regular.
  Since \(b\) is not a zero-divisor for~\(M\), this means that \(\aa'\) is not regular for~\(M''=M/b\).
  
  We write \(T''=T(b)\) and define \(R''\) and~\(\pi''\) accordingly.
  Again by Lemma\nobreakspace \ref {thm:localizing-quotient}, \(\pi''(S)\) is \((k-1)\)-localizing for~\(X\), considered as a \(T''\)-space.
  Appealing once more to Lemma\nobreakspace \ref {thm:reduction}, we get an isomorphism
  \begin{equation}
    M'' \cong H_{T''}(X).
  \end{equation}
  Given that \(\aa'\) is not \(M''\)-regular, \(M''\) cannot be a \((k-1)\)-st syzygy over~\(R''\).
  By induction, we can therefore find a sequence~\(\pi''(\cc)\) of length at most~\(k-1\) in~\(\pi''(S)\) that is regular for~\(R''\),
  but not for~\(M''\). Thus, \((b,\cc)\) is an \(R\)-regular sequence in~\(S\) of length at most~\(k\)
  that is not \(M\)-regular, as desired.
\end{proof}

\section{Big polygon spaces}
\label{sec:big-polygon}

Let \(r\ge1\). We write \([r]=\{1,\dots,r\}\)
and \(\Delta\) for the simplex with vertex set~\([r]\),
considered as a simplicial complex.
We call a length vector~\(\ell\in\R^{r}\) \emph{strongly generic} if \(\lambda(\sigma)\ne\lambda(\tau)\)
for any two distinct simplices~\(\sigma\),~\(\tau\) in~\(\Delta\),
where
\begin{equation}
  \lambda(\sigma) = \sum_{j\in\sigma}\ell_{j}.
\end{equation}

Two generic length vectors are called \newterm{equivalent}
if they induce the same notion of `long' and `short' on subsets of~\([r]\).
The equivalence classes of generic length vectors~\(\ell\) are open polyhedral cones
in~\(\R^{r}\) which are the connected components of the complement of a hyperplane arrangement.
Because strong genericity means that
certain additional hyperplanes are avoided, any generic length vector is equivalent to a strongly generic one.
Two equivalent generic length vectors give rise to equivariantly diffeomorphic big polygon spaces,
hence to isomorphic equivariant cohomologies.
Moreover, there is no loss of generality
if one assumes \(\ell\) to be positive and weakly increasing,
see~\cite[Sec.~2]{Franz:maximal}.
In this case, non-equivalent generic length vectors give rise to big polygon spaces
which even non-equivariantly are not diffeomorphic \cite[Prop.~3.7]{Franz:maximal}.
For the rest of this section,
\(\ell\in\R^{r}\) denotes a strongly generic length vector with positive coordinates.

For any \(R\)-algebra~\(\bar R\), we write \(\CC(\Delta;\bar R)\) for the Koszul complex with coefficients in~\(\bar R\).
That is, \(\CC(\Delta;\bar R)\) is a free \(\bar R\)-module with basis~\(\Delta\) and differential
\begin{equation}
  \label{eq:Koszul-d}
  d\gamma = \sum_{j\in\gamma}\pm t_{j}^{q}\,(\gamma\setminus j)
\end{equation}
for~\(\gamma\in\Delta\), \cf~\cite[Sec.~5]{Franz:maximal}.
(Note that we sometimes omit braces, as in~\(\gamma\setminus j\).)
We introduce a grading by giving each generator~\(t_{i}\in R\) the degree~\(2\) and each~\(\gamma\in\Delta\)
the degree~\((2p+2q-1)\cdot\#\gamma\).
The differential~\eqref{eq:Koszul-d} then has degree~\(1-2p\).

Let \(S\subset\Delta\) be a subset. We define \(S_{+}\) and~\(S_{-}\) to be the set
of \(\ell\)-long and \(\ell\)-short simplices in~\(S\), respectively.
We write \(\CC(S;\bar R)\) for the \(\bar R\)-submodule of~\(\CC(\Delta;\bar R)\) with basis~\(S\)
so that
\begin{equation}
  \label{eq:direct-sum}
  \CC(\Delta;\bar R) = \CC(S;\bar R) \oplus \CC(\Delta\setminus S;\bar R)
\end{equation}
as \(\bar R\)-modules. If \(S\) is a simplicial subcomplex of~\(\Delta\), then \(\CC(S;\bar R)\) and \(\CC(S_{-};\bar R)\)
are subcomplexes of~\(\CC(\Delta;\bar R)\), but \(\CC(S_{+};\bar R)\) is not in general.
For any~\(S\) we define the subcomplex
\begin{equation}
  \DD(S;\bar R) = \CC(S;\bar R) + d\,\CC(S;\bar R)
  \subset \CC(\Delta;\bar R).
\end{equation}

For any~\(c=\sum_{\sigma\in\Delta} c_{\sigma}\sigma\in\CC(\Delta;\bar R)\), we write
\begin{equation}
  \supp c = \{\,\sigma\in\Delta \,|\, c_{\sigma}\ne 0\,\}
\end{equation}
for its support and, assuming \(c\notin\CC(\Delta_{+})\),
\begin{equation}
  \lambda(c) = \min\{\,\lambda(\sigma) \,|\, \text{\(\sigma\in\supp c\) short} \,\}. 
\end{equation}

\begin{lemma}
  \label{thm:desc-HTX}
  Consider the differential as a map
  \begin{equation*}
    f_{\ell}\colon \CC(\Delta_{+};R) \to \CC(\Delta;R)\!\bigm/\! \CC(\Delta_{+};R) \cong \CC(\Delta_{-};R),
    \qquad
    \gamma\mapsto d\gamma.
  \end{equation*}
  Then there is a short exact sequence of graded \(R\)-modules
  \let\longrightarrow\to
  \begin{equation*}
    0 \longrightarrow
    \coker f_{\ell} \longrightarrow
    \HT^{*}(X(\ell)) \longrightarrow
    (\ker f_{\ell})[-2p] \longrightarrow
    0.
  \end{equation*}
  In particular, the syzygy order of~\(\HT^{*}(X(\ell))\) over~\(R\)
  equals that of~\(\coker f_{\ell}\).
\end{lemma}

Here ``\([-2p]\)'' denotes a degree shift by~\(2p\) downwards.
  The sequence actually splits by a result of Puppe~\cite[Lemma~3.12]{Puppe:2018}.

\begin{proof}
  See \cite[Sec.~4, Lemma~6.2]{Franz:maximal}. Note that we have indexed the basis elements
  in a form more convenient for our purposes.
\end{proof}

For any~\(\gamma\in\Delta_{+}\) we define
\begin{equation}
  \sigma_{\ell}(\gamma) = \# \{\, j\in\gamma \mid \text{\(\gamma\setminus j\) is \(\ell\)-short}\,\}
\end{equation}
and
\begin{equation}
  \mu(\ell) = \min\bigl\{\,\sigma_{\ell}(\gamma) \mid \text{\(\gamma\in\Delta_{+}\) and \(\sigma_{\ell}(\gamma)>0\)}\,\bigr\} \ge 1
\end{equation}
as in~\cite[eqs.~(6.6)--(6.7)]{Franz:maximal}.

\begin{theorem}
  \label{thm:syzord-mu-ell}
  The syzygy order of~\(\HT^{*}(X(\ell))\) over~\(R\) is \(\mu(\ell)-1\).
\end{theorem}

In~\cite[Prop.~6.3]{Franz:maximal} it is shown that \(\mu(\ell)-1\) is an upper bound for the syzygy order,
and it was conjectured that one has equality \cite[Conj.~6.6]{Franz:maximal}.

\begin{proof}
  According to~Lemma\nobreakspace \ref {thm:desc-HTX},
  the syzygy order of~\(\HT^{*}(X(\ell))\) equals that of
  \begin{equation}
    M(\ell)=\CC(\Delta;R) \!\bigm/\! \DD(\Delta_{+};R).
  \end{equation}
  By what we have just said,
  we only have to show that \(M(\ell)\) is a syzygy of order at least~\(\mu(\ell)-1\).
  
  The isotropy subgroups appearing in~\(X(\ell)=X_{p,q}(\ell)\) are the coordinate subtori of~\(T=(S^{1})^{r}\).
  Hence for any~\(k\) the set~\(S=\{t_{1},\dots,t_{r}\}\subset\RZtwo\) is \(k\)-localizing for~\(X\).
  By  Theorem\nobreakspace \ref {thm:char-syzygy}, it suffices to show that 
  for any~\(k<\mu(\ell)\) and any pairwise distinct elements~\(i_{1}\),~\dots,~\(i_{k}\)
  the sequence~\((t_{i_{1}},\dots,t_{i_{k}})\) is \(M(\ell)\)-regular.

\def\ti{t_{i}^{q}}
  We proceed by induction on~\(k\), the case~\(k=0\) being void. For~\(k>0\), we know by induction
  that the sequence~\(t_{i_{1}}\),~\dots,~\(t_{i_{k-1}}\) is \(M(\ell)\)-regular.
  It remains to show that \(t_{i_{k}}\) is not a zero-divisor
  in~\(N=M(\ell)/(t_{i_{1}},\dots,t_{i_{k-1}})\)
  or, equivalently, that \(t_{i_{k}}^{q}\) is not a zero-divisor in~\(N\).
  (Recall that \(M(\ell)\) is graded and bounded below, so that \(t_{i_{k}}N\ne N\).)
  We write \(I=\{i_{1},\dots,i_{k-1}\}\) and~\(i=i_{k}\).

  We start by observing that
  \begin{equation}
  \begin{split}
    N &= \CC(\Delta;R) \!\bigm/\! \bigl(\,\DD(\Delta_{+};R)+(t_{i_{1}},\dots,t_{i_{k-1}})\,\CC(\Delta;R)\,\bigr) \\
    &= \CC(\Delta;\bar R) \,/\, \DD(\Delta_{+};\bar R)      
  \end{split}
  \end{equation}
  where
  \begin{equation}
    \bar R = R/(t_{i_{1}},\dots,t_{i_{k-1}}) =
    \kk\bigl[\,t_{j}\mid j\notin I\,\bigr].
  \end{equation}
  In this case the differential~\eqref{eq:Koszul-d} takes the form
  \begin{equation}
    \label{eq:Koszul-d-new}
    d\gamma = \sum_{\!j\in\gamma\setminus I\!}\pm\,t_{j}^{q}\,(\gamma\setminus j).
  \end{equation}

  Assume that our claim is false. Then there is a~\(c\in\CC(\Delta;\R)\)
  such that \(\ti c\) is contained in~\(\DD(\Delta_{+};\bar R)\) while \(c\) itself is not.
  We can write \(\ti c= a + db\) for some~\(a\),~\(b\in\CC(\Delta_{+};\bar R)\).
  Since
  \begin{equation}
    (a+\ti a')+d(b+\ti b') = \ti (c + a' + db'),
  \end{equation}
  we may assume \(a\) and~\(b\) to be \emph{\(\ti\)-free}.
  By this we mean that in the canonical monomial basis for~\(\bar R\),
  no monomial divisible by~\(\ti\) appears in the non-zero coefficients~\(a_{\gamma}\) and~\(b_{\gamma}\) of~\(a\) and~\(b\)
  with respect to the basis~\(\Delta\) of~\(\CC(\Delta;\bar R)\).
  It implies that \(c\) is \(\ti\)-free.

  Let \(\Gamma\subset\Delta\) be the simplicial complex of the facet not containing \(i\), so that \(\Delta\) is the cone over~\(\Gamma\) with vertex~\(i\).
  Recall that we have a bijection between the simplices~\(\gamma\in\Gamma\) and those in~\(\Delta\setminus\Gamma\)
  given by~\(\gamma\mapsto\hatgamma\coloneqq\gamma\cup i\).

  Using~\eqref{eq:direct-sum}, we can decompose \(a\),~\(b\) and~\(c\) as
  \begin{align}
    c = c_{1}+c_{2} &\in \CC(\Gamma;\bar R)\oplus\CC(\Delta\setminus \Gamma;\bar R), \\
    a = a_{1}+a_{2},\; b=b_{1}+b_{2} &\in \CC(\Gamma_{+};\bar R)\oplus\CC(\Delta_{+}\setminus \Gamma_{+};\bar R)
  \end{align}
  By inspection of~\eqref{eq:Koszul-d-new} we see that \(db_{1}\in\CC(\Gamma;\bar R)\) is \(\ti\)-free and
  \begin{align}
    db_{2} = e_{1} + e_{2} \in \CC(\Gamma;\bar R)\oplus\CC(\Delta\setminus \Gamma;\bar R)
  \end{align}
  where \(e_{1}\) is divisible by~\(\ti\) and \(e_{2}\) is \(\ti\)-free. Hence
  \begin{equation}
    \ti c_{1} = \underbrace{a_{1}+db_{1}}_{\text{\(\ti\)-free}} + \underbrace{e_{1}}_{\!\!\!\!\text{div.\ by \(\ti \)}\!\!\!\!}
    \quad\text{and}\quad
    \ti c_{2} = \underbrace{a_{2}+e_{2}}_{\text{\(\ti\)-free}}.
  \end{equation}
  This implies
  \begin{equation}
    \ti c_{1} = e_{1},
    \qquad
    a_{1}+db_{1} = 0,
    \qquad
    c_{2} = a_{2}+e_{2} = 0.
  \end{equation}
  Hence \(c\in\CC(\Gamma;\bar R)\), and we can write it in the form
  \begin{equation}
    \label{eq:decomposition-ti}
    \ti c = a+ db,
    \qquad
    a,b\in \CC(\Delta_{+}\setminus \Gamma_{+};\bar R).
  \end{equation}

  We additionally assume that \(c\) is a counterexample maximizing~\(\lambda(c)\).
  The simplex~\(\sigma\in\supp c\) realizing \(\lambda(c)\) is necessarily short
  since \(c\notin\CC(\Delta_{+};\bar R)\subset\DD(\Delta_{+};\bar R)\).
  We finally require that among all these counterexamples we pick one
  with the fewest monomials appearing in~\(c_{\sigma}\in\bar R\).
  
  Figure\nobreakspace \ref {fig:beta} may help the reader to visualize the simplices constructed in the following arguments.

  \begin{figure}
    \begin{tikzpicture}
      \draw (0,2) node[above]{\(i\)} -- (1.5,0) node[right]{\(j'\)} -- node[auto]{\(\tau\)} (0,-1) node[below]{\(\rho\)} -- node[auto]{\(\sigma\)} (-1.5,0) node[left]{\(j\)} -- cycle;
      \draw (0,2) -- node[left]{\(\hat\rho\)} (0,-1);
      \draw[dashed] (-1.5,0) -- (1.5,0);
      \draw[->, bend right=30] (-1.5,1.4) node[above]{\(\hat\sigma\)} to (-0.75,0.5);
      \draw[->, bend left=30] (1.5,1.4) node[above]{\(\hat\tau\)} to (0.75,0.5);
      \draw[->, bend left=30] (1.15,-1.15) node[right]{\(\beta\)} to (0.25,-0.5);
    \end{tikzpicture}
    \caption{\label{fig:beta}The simplex~\(\hat\beta\). The face~\(\rho\) 
      is of codimension~\(3\).}
  \end{figure}
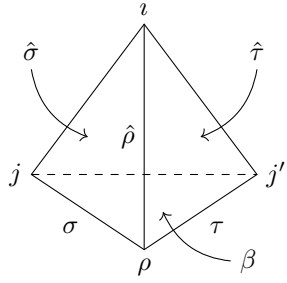
  
  Being short, \(\sigma\) cannot be contained in~\(\supp a\).
  Hence \(\hat\sigma\in\supp b\) and \(b_{\hat\sigma}=\pm c_{\sigma}\) by~\eqref{eq:decomposition-ti} and~\eqref{eq:Koszul-d-new}.
  In particular, \(\hat\sigma\) is long.
  Since \(\sigma\) is short and \(\mu(\ell)\ge k+1\), we conclude that \(\hat\sigma\) has \(k+1\)~short facets. 
  Hence there is a short facet of the form
  \begin{equation}
    \label{eq:def-rho}
    \hat\rho = \hat\sigma\setminus j,
    \quad
    j\notin I\cup i
  \end{equation}
  given that \(\#I=k-1\).
  Let us write
  \begin{equation}
    \label{eq:vertices-sigma}
    \sigma=\{j_{0},\dots,j_{m}\}
    \quad\text{with}\quad
    \ell_{j_{0}}<\dots<\ell_{j_{m}}
  \end{equation}
  (where we have used our assumption that \(\ell\) is strongly generic) and
  \begin{equation}
    \label{eq:def-j-p}
    j_{p}=\max(\sigma\setminus I).
  \end{equation}
  By~\eqref{eq:vertices-sigma}, we may assume \(j=j_{p}\) in~\eqref{eq:def-rho}
  because replacing \(j\) by~\(j_{p}\) can only decrease the length of~\(\hat\rho\).

  Looking at~\eqref{eq:Koszul-d-new},
  we have \(\hat\rho\in\supp d\hat\sigma\) since \(j_{p}\notin I\).
  Given that \(\hat\rho\) is short and not contained in~\(\Gamma\),
  it cannot appear in~\(db=\ti c-a\). Hence
  there must be a (necessarily long) simplex~\(\hat\tau\ne\hat\sigma\) appearing in~\(b\) and
  having \(\hat\rho\) as a facet.

  We have
  \begin{equation}
    \hat\tau=\hat\rho\cup j'
    \quad\text{for some~\(j'\notin\hat\sigma\).}
  \end{equation}
  The contribution of~\(d(b_{\hat\sigma}\hat\sigma)=\pm d(c_{\sigma}\hat\sigma)\)
  to the coefficient~\((db)_{\hat\rho}\) of~\(\hat\rho\) in~\(db\)
  is \(\pm t_{j_{p}}^{q} c_{\sigma}\), and that of~\(d(b_{\hat\tau}\hat\tau)\) likewise is \(\pm t_{j'}^{q} b_{\hat\tau}\).
  Since all monomials appearing in~\((dc_{\sigma}\hat\sigma)_{\hat\rho}\) must somehow
  be compensated for by other simplices appearing in~\(b\),
  we may choose \(\hat\tau\) such that \(t_{j'}^{q}\) divides a monomial appearing in~\(c_{\sigma}\).
  Because \(\tau\) appears in~\(db\) and \(\sigma\) is the shortest simplex appearing there,
  we additionally have \(\lambda(\tau)>\lambda(\sigma)\) or,
  in other words, \(\ell_{j'}>\ell_{j_{p}}\),
  again by strong genericity.

  Now \(\hat\tau\) is a long facet of
  \begin{equation}
    \hat\beta=\hat\sigma\cup j'
    \quad
    \text{where \(\beta=\sigma\cup j'\)}.
  \end{equation}
  The other facets of~\(\hat\beta\) different from~\(\beta\)
  are obtained from~\(\hat\tau\) by substituting \(j_{p}\) for some~\(j_{q}\in\tau\).
  If \(q<p\), we get another long facet by~\eqref{eq:vertices-sigma}.
  Hence \(\hat\beta\) has at most~\(\# I+1=k\) short facets by~\eqref{eq:def-j-p},
  including possibly \(\beta\).
  But \(\mu(\ell)\ge k+1\), so that all facets of~\(\hat\beta\) are long. 

  Since we have
  \begin{equation}
    0 = d\,d\,\hat\beta = d\Bigl(\pm\ti\,\beta\pm\sum_{j\in\hat\beta\setminus i}\! t_{j}^{q}\,(\hat\beta\setminus j)\Bigr),
  \end{equation}
  we can write
  \( 
    \ti\,\tilde c = d\,\tilde b
  \) 
  with
  \begin{align}
    \label{eq:def-tilde-c}
    \tilde c &= d\,\beta = \sum_{\!\!j\in\beta\setminus I\!\!}\pm t_{j}^{q}\,(\beta\setminus j)\in \CC(\Gamma;\bar R)
    \shortintertext{and}
    \tilde b &= \pm\sum_{\!\!j\in\beta\setminus I\!\!} t_{j}^{q}\,(\hat\beta\setminus j) \in \CC(\Delta_{+}\setminus\Gamma_{+};\bar R).
  \end{align}
  
  Consider now all monomials appearing in~\(c_{\sigma}\) that are divisible by~\(t_{j'}^{q}\)
  and write their sum as \(t_{j'}^{q} x\)
  with~\(x\in\bar R\). Then \(x\ne0\) by our choice of~\(\hat\tau\), and
  no monomial appearing in it is divisible by~\(\ti\) since \(c\) is \(\ti\)-free.
  The preceding discussion implies
  \begin{equation}
    \ti(c+x\,\tilde c) = a+d(b+x\,\tilde b)
  \end{equation}
  where both~\(a\) and~\(b+x\,\tilde b\in\CC(\Delta_{+}\setminus\Gamma_{+};\bar R)\) are \(\ti\)-free
  and \(c+x\,\tilde c\in\CC(\Gamma;\bar R)\).
  In particular, \(c+x\,\tilde c\) is another counterexample of the form~\eqref{eq:decomposition-ti}
  to our claim that \(\ti\) is not a zero-divisor in~\(M(\ell)\).

  Since \(\ell_{j'}>\ell_{j_{p}}\) and \(j_{q}\in I\) for~\(p<q\le m\),
  the simplex~\(\sigma=\beta\setminus j'\) is the shortest one appearing in the sum~\eqref{eq:def-tilde-c}.
  Hence
  \begin{equation}
    \label{eq:length-new-c}
    \lambda(c+x\,\tilde c)\ge\lambda(c).
  \end{equation}
  The coefficient of~\(\sigma\) in~\(c+x\,\tilde c\) is of the form
  \begin{equation}
    (c+x\,\tilde c)_{\sigma}=c_{\sigma}-t_{j'}^{q} x.
  \end{equation}
  If it vanishes, then we have a strict inequality in~\eqref{eq:length-new-c}
  since \(\ell\) is strongly generic. This
  would contradict our choice of~\(c\) with maximal~\(\lambda(c)\).
  If it does not vanish, then it is still obtained
  from~\(c_{\sigma}\) be removing certain monomials.
  As such, it contains fewer monomials than~\(c_{\sigma}\),
  again contradicting our choice of~\(c\).
  
  We conclude that no counterexample exists.
\end{proof}

\section{Classification of high syzygies}
\label{sec:high-syzygies}

Using the result of the previous section, we can extend the classification of
big polygon spaces with high syzygies in their equivariant cohomology.
Throughout this section, \(\ell\in\R^{r}\) denotes a generic length vector with positive and weakly increasing coefficients.

Maximal syzygies, that is, those of order~\(m\) for \(r=2m+1\)~odd or \(r=2m+2\)~even
were determined in~\cite{Franz:maximal}.
We are going to rephrase the proof in our setting and
extend the result to syzygies of order~\(m-1\).

\begin{lemma}
  \label{thm:J-mu-ell}
  If there is a long subset~\(J\subset[r]\) of size~\(\#J=\mu(\ell)\), then
  \begin{equation*}
    \ell\sim(0,\dots,0,\underbrace{1,\dots,1}_{2\mu(\ell)-1}).
  \end{equation*}
\end{lemma}

\begin{proof}
  \def\jmin{j_{\mathrm{min}}}
  \def\jmax{j_{\mathrm{max}}}
  We may assume \(\ell\) to be strongly generic and set \(\mu(\ell)=k\).
  Note that all subsets~\(I\subset[r]\) with fewer than \(k\)~elements are short
  for otherwise we would get the contradiction
  \begin{equation}
    k = \mu(\ell)\le\sigma_{\ell}(I) = \#I < k
  \end{equation}
  for a (necessarily nonempty) long set~\(I\) of minimal size. 

  Among all long subsets~\(J\subset[r]\) of size~\(k\),
  we pick the one with minimal~\(\ell(J)\).
  We set \(\jmin=\min(J)\) and \(\jmax=\max(J)\).
  By what we have just said, \(J\setminus\jmax\) is short. 

  Let \(I\subset[r]\) be the set of those values~\(j\notin J\setminus\jmax\) such that \((J\setminus\jmax)\cup j\) is long.
  This set contains \(\jmax\) and therefore is non-empty. Hence \(\#I\ge k\)
  and \(\#(I\setminus\jmax)\ge k-1\).

  For any~\(i\in I\setminus\jmax\), the set~\(J_{i}=(J\setminus\jmax)\cup i\) is long and of size~\(k\), and \(J_{i}\setminus i\) is short.
  Hence
  \begin{equation}
    k \le \sigma_{\ell}(J_{i}) \le \# J_{i} = k.
  \end{equation}
  This implies \(\ell(J_{i})>\ell(J)\) for otherwise \(\ell(J)\) would not be minimal.
  In summary, \(I\) consists of~\(\jmax\) and \(k-1\) values larger than~\(\jmax\).

  Consider the remaining \(r-\#(J\cup I)=r-2k+1\)~elements of~\([r]\setminus(J\cup I)\).
  If one of them were greater than~\(\jmin\), then we would have \(\ell([r]\setminus J)>\ell(J)\),
  contradicting the assumption that \(J\) is long. These elements therefore are smaller than~\(\jmin\),
  and we conclude that
  \begin{equation}
    J=\{\jmin,\dots,\jmax\}=\{r-2k+2,\dots,r-k+1\}.
  \end{equation}

  This implies that any subset of~\([r]\) containing \(k\)~values~\(\ge\jmin=r-2k+2\) is long.
  These sets are exactly the long sets for the length vector
  \begin{equation}
    \ell' = \bigr(\underbrace{0,\dots,0}_{r-2k+1},\underbrace{1,\dots,1}_{2k-1}\,\bigl).
  \end{equation}
  We conclude that they comprise half of all subsets and therefore that
  \(\ell\) and~\(\ell'\) induce the same notion of `long' and `short'.
\end{proof}

\begin{proposition}[{\cite[Cor.~6.4]{Franz:maximal}}]
  \label{thm:syzymy-m-characterization}
  Let \(r\ge1\).
  \begin{enumerate}
  \item Assume that \(r=2m+1\) is odd. Then \(\HT^{*}(X(\ell))\) is a syzygy of order~\(m\) if and only if
    \(\ell\sim(1,\dots,1)\).
  \item Assume that \(r=2m+2\) is even. Then \(\HT^{*}(X(\ell))\) is a syzygy of order~\(m\) if and only if
    \(\ell\sim(0,1,\dots,1)\).
  \end{enumerate}
\end{proposition}

\begin{proof}
  By Theorem\nobreakspace \ref {thm:syzord-mu-ell}, the condition on the syzygy order
  translates into~\(\mu(\ell)=m+1\).  
  In both cases it is immediate to check
  that this is satisfied by the given length vectors.
  It remains to show the `only if' direction.

  As in the previous proof, any subset with fewer than \(\mu(\ell)=m+1\)~elements is short.
  Hence there must be a long subset of size~\(m+1\) for otherwise 
  more than half of all subsets would be short.
  The claim now follows from Lemma\nobreakspace \ref {thm:J-mu-ell}.
\end{proof}

\begin{proposition}
  Let \(r\ge3\).
  \begin{enumerate}
  \item Assume that \(r=2m+1\) is odd. Then \(\HT^{*}(X(\ell))\) is a syzygy of order~\(m-1\) if and only if
    \(\ell\sim(0,0,1,\dots,1)\).
  \item Assume that \(r=2m+2\) is even. Then \(\HT^{*}(X(\ell))\) is a syzygy of order~\(m-1\) if and only if
    \(\ell\sim(0,0,0,1,\dots,1)\) or~\(\ell\sim(1,1,1,2,\dots,2)\).
  \end{enumerate}
\end{proposition}

We can restrict ourselves to~\(r\ge3\) here since \(\HT(X(\ell))\) is always torsion for~\(r\le2\).

\begin{proof}
  This time the condition on the syzygy order translates into~\(\mu(\ell)=m\).  
  In all cases it is elementary to verify
  that it is satisfied by the given length vectors.
  It remains to show the `only if' direction.

  Let \(J\subset[r]\) be a long subset of minimal size.
  Since half of all subsets are long, we have \(1\le\#J\le m+1\)
  and also \(m=\mu(\ell)\le\sigma_{\ell}(J)=\#J\),
  hence \(m\le\#J\le m+1\).

  Assume that \(r=2m+1\) is odd. If \(\#J=m+1\),
  then all subsets of size at most~\(m\) are short.
  Since these are already half of all subsets,
  those having at least \(m+1\)~elements are long. This
  implies \(\mu(\ell)=m+1\), contrary to our assumption. Hence \(\#J=m=\mu(\ell)\),
  and we can appeal to Lemma\nobreakspace \ref {thm:J-mu-ell}.

  Now let \(r=2m+2\) be even.
  The case~\(\#J=m\) is dealt with as before.
  So we can assume that long sets have at least \(m+1\)~elements,
  and we have to show \(\ell\sim\ell'=(1,1,1,2,\dots,2)\).
  
  For the purpose of this proof, call a subset~\(J\subset[r]\) \emph{distinguished}
  if it is of size~\(m+1\) and contains \(2\) and~\(3\).
  Assume that there is a distinguished long set~\(J\).
  We claim that in this case all distinguished sets are long.

  In order to prove this, choose a~\(j\notin J\) such that~\(j>3\).
  (This is possible because there are \(m+1\ge2\) elements not in~\(J\).)
  Then both \(J'=J\cup j\) and \(J=J'\setminus j\) are long, hence so are
  \(J'\setminus 2\) and \(J'\setminus 3\). Thus,
  \begin{equation}
    \sigma_{\ell}(J')\le (m+2)-3=m-1.
  \end{equation}
  Since \(\mu(\ell)=m\), this implies \(\sigma_{\ell}(J')=0\). In other words,
  replacing any element of~\(J\) by an element~\(J\niton j>3\) leads to another 
  long set. Applying this procedure repeatedly, we can transform \(J\) into any other
  distinguished set while keeping it long, which proves the claim.

  Given a distinguished set, we can also replace \(2\) and~\(3\) by larger elements
  without making the set short.
  So we see that any subset~\(J\subset[r]\) of size~\(m+1\) not containing~\(1\)
  is long, as are all subsets of larger size (because their complements,
  being of size at most~\(m\), are short).
  These sets are exactly the long subsets for~\(\ell''=(0,1,\dots,1)\) and therefore
  all \(\ell\)-long subset. But this is impossible as \(\mu(\ell'')=m+1\ne m\).

  We conclude that any distinguished set is short. Hence so is any subset~\(J\)
  of size~\(m+1\) such that \(\#(J\cap\{1,2,3\})\ge2\). Together with the
  subsets of smaller size, these are exactly the \(\ell'\)-short subsets.
  So they are also exactly the \(\ell\)-short subsets, which shows
  \(\ell\sim\ell'\).
\end{proof}

There seems to be no easy description of syzygies of lower order.
For instance, computer experiments show that
for~\(r=9=2\cdot4+1\) there are, up to equivalence, \(5\)~length vectors~\(\ell\)
(out of~\(175,428\)) such that \(\HT^{*}(X(\ell))\) has syzygy order~\(2=4-2\),
and for~\(r=10=2\cdot4+2\) there are \(18\) (out of~\(52,980,624\)).
(The numbers of non-\allowbreak equivalent length vectors can be found in~\cite[Sec.~10.3.1]{Hausmann:mod2}.)

\section{The real case}
\label{sec:real}

Let \(\ell\in\R^{r}\) be a generic length vector, and let \(p\),~\(q\ge1\).
The \emph{real big polygon space}~\(Y(\ell)=Y_{p,q}(\ell)\) is defined by restricting all variables
to the reals, that is, by
\begin{equation}
  \left\{\:
  \begin{aligned}
  \|u_{j}\|^{2} + \|z_{j}\|^{2} &= 1 && \text{for any~\(1\le j\le r\),} \\
  \ell_{1}u_{1}+ \dots + \ell_{r}u_{r} &= 0\mathrlap{,}
  \end{aligned}
  \right.
\end{equation}
where~\(u_{1}\),~\dots,~\(u_{r}\in\R^{p}\) and~\(z_{1}\),~\dots,~\(z_{r}\in\R^{q}\).
It is the fixed point set of~\(X(\ell)\) under complex conjugation of all variables,
hence again smooth. The \(2\)-torus~\(G=(\Z_{2})^{r}\) of rank~\(r\) acts on~\(Y(\ell)\) by reversing the signs
of the \(z\)-variables.
We assume now that \(\kk\) is a field of characteristic~\(2\).
The \(G\)-equivariant cohomology of
real big polygon spaces (and more general spaces)
with coefficients in~\(\kk\)
has been studied by Puppe~\cite{Puppe:2018},
following the work of Hausmann~\cite[Sec.~10.3]{Hausmann:mod2}.

The fixed point set~\(Y(\ell)^{G}\) is empty for~\(p=1\), by the very definition
of a generic length vector.
By the localization theorem, this implies that \(\HG^{*}(Y(\ell))\) is a torsion module
over~\(R=H^{*}(BG)\) in this case.
For~\(p>1\), however, the theory parallels the one
for the complex case. In particular,
the equivariant cohomology~\(\HG^{*}(Y(\ell))\) is given by a formula analogous to Lemma\nobreakspace \ref {thm:desc-HTX},
which allows us to proceed in the same way as before. We content ourselves with sketching the arguments.

\begin{lemma}
  Grade \(\CC(\Delta;R)\) by setting \(\deg{t_{i}}=1\) for each generator~\(t_{i}\in R\)
  and \(\deg{\gamma}=(p+q-1)\cdot\#\gamma\) for each~\(\gamma\in\Delta\).
  With \(f_{\ell}\) defined as in Lemma\nobreakspace \ref {thm:desc-HTX},
  there is a short exact sequence of graded \(R\)-modules
  \let\longrightarrow\to
  \begin{equation*}
    0 \longrightarrow
    \coker f_{\ell} \longrightarrow
    \HT^{*}(X(\ell)) \longrightarrow
    (\ker f_{\ell})[-p] \longrightarrow
    0.
  \end{equation*}
  In particular, the syzygy order of~\(\HT^{*}(X(\ell))\) over~\(R\)
  equals that of~\(\coker f_{\ell}\).
\end{lemma}

\begin{proof}
  See~\cite[Sec.~3]{Puppe:2018}, in particular Proposition~3.11 and equations~(3.19)--(3.22).
  Alternatively, one could adapt the proof given in~\cite{Franz:maximal}.
\end{proof}

Let \(X\) be a \(G\)-space, and for~\(x\in X\)
let \(p_{x}\colon H^{1}(BG)\to H^{1}(BG_{x})\) be the restriction map.
For~\(0\le k\le r\),
we call a subset~\(S\subset H^{1}(BG)\) \newterm{\(k\)-localizing} if for any~\(x\in X\)
at least \(\min(k, \corank G_{x})\) linearly independent elements from~\(S\)
lie in~\(\ker p_{x}\).

We need the following analogue of Lemma\nobreakspace \ref {thm:syzygy-lin-indep-seq} from~\cite{AlldayFranzPuppe:orbits2}.
A \newterm{nice} \(G\)-space is defined in analogy with the torus case.

\begin{lemma}
  Let \(X\) be a nice \(G\)-space and let \(k\ge0\). Then \(\HG^{*}(X)\) is a \(k\)-th syzygy
  if and only if every linear independent sequence in~\(H^{1}(BG;\F_{2})\) of length at most~\(k\)
  is \(\HG^{*}(X)\)-regular.
\end{lemma}

Arguing as in Sections~\ref{sec:char-syz} and~\ref{sec:big-polygon}, we obtain the following.

\begin{theorem}
  \label{thm:char-syzygy-2}
  Let \(X\) be a nice \(G\)-space, and
  let \(S\subset H^{1}(BG;\F_{2})\) be \(k\)-localizing for~\(X\) for some~\(k\ge 0\).
  Then \(\HG^{*}(X)\) is a \(k\)-th syzygy over~\(R\)
  if and only if any linearly independent 
  sequence in~\(S\) of length at most~\(k\) is
  \(\HG^{*}(X)\)-regular.
\end{theorem}

\begin{theorem}
  \label{thm:syzord-mu-ell-2}
  In the case~\(p>1\)
  the syzygy order of~\(\HG^{*}(Y(\ell))\) over~\(R\) is \(\mu(\ell)-1\).
  In particular,
  it is the same as the syzygy order of~\(\HT(X(\ell);\Q)\) over~\(H^{*}(BT;\Q)\).
\end{theorem}

As a consequence, the characterizations of high syzygies derived in Section\nobreakspace \ref {sec:high-syzygies}
carry over to real big polygon spaces.
The analogue of Proposition\nobreakspace \ref {thm:syzymy-m-characterization}
has already been established by Puppe~\cite[Prop.~3.14]{Puppe:2018}.

\end{document}